\documentclass[11pt,psamsfonts,fleqn,leqno]{amsart}
\usepackage{amsfonts}
\usepackage{mathrsfs}
\usepackage{a4wide,amssymb,amsmath,amsthm,xspace,epsfig,color}
\usepackage{graphicx,ifthen}

\def\L{\mathbb L}

\def\R{\mathbb R}
\def\S{\mathbb S}

\def\al{\alpha}
\def\be{\beta}
\def\ga{\gamma}

\def\th{\theta}

\def\si{\sigma}
\def\ph{\varphi}
\def\om{\omega}
\def\c{\mathfrak c}
\def\*{\times}
\def\la{\langle}
\def\ra{\rangle}
\def\ds{\displaystyle}
\def\mc{\mathcal}
\newcommand{\Int}[1]{\mathaccent 23{#1}}
\newcommand{\ol}{\overline}
\newcommand{\sm}{\setminus}

\newtheorem{thm}{Theorem}

\newtheorem{cor}[thm]{Corollary}

\newtheorem{lem}[thm]{Lemma}
\newtheorem{prop}[thm]{Proposition}

\newtheorem{claim}[thm]{Claim}

\begin{document}

\title{Exotic Differential Structures in Dimension 2}
\author[Sunanda Dikshit]{Sunanda Dikshit$^1$}
\address{$^1$ Department of Mathematics, The University of
  Auckland, Private Bag 92019, Auckland, New Zealand}
\email{s.dixit@math.auckland.ac.nz}

\author[David Gauld]{David Gauld$^2$}
\address{$^2$Department of Mathematics, The University of
  Auckland, Private Bag 92019, Auckland, New Zealand}
  \email{d.gauld@auckland.ac.nz}

\thanks{$^{1,2}$Supported in part by a grant from the New Zealand Institute of Mathematics and its Applications. This work represents part of the first author's PhD thesis prepared at the University of Auckland under the supervision of the second author.}

\date{\today}

\begin{abstract}
It is known that the long line supports $2^{\aleph_1}$ many non-diffeomorphic differential structures. We show that the long plane supports a similar number of exotic differential structures, ie structures which are not merely diffeomorphic to the product of two structures on the factor spaces.
\end{abstract}
\maketitle


 
  \section{Introduction}
 
In this paper, by a \emph{differential structure} we mean a C$^r$ differential structure for any $r\ge 1$. Recall, \cite{KP68}, that every C$^r$ structure contains a C$^s$ structure for any $s$ satisfying $r<s\le\infty$. Hence we are not concerned about the value of $r$.

It is well-known that euclidean space, $\R^n$, possesses a unique differential structure up to diffeomorphism for $n\not=4$ but $\R^4$ has $\c$ mutually non-diffeomorphic differential structures; see \cite[page 95]{K89} for early details and \cite{M11} for a recent survey. Thus $\R^4$ possesses many \emph{exotic differential structures}, i.e., differential structures which are not diffeomorphic to the 4-fold product of $\R$ with the usual structure (or the 2-fold product of $\R^2$ with the usual structure). Of course exotic differential structures were discovered more than half a century ago by Milnor in \cite{M56} where there is given the first construction of a differential structure on the 7-sphere $\S^7$ which is not diffeomorphic to the usual (product) differential structure inherited from $\R^8$. The existence of two mutually non-diffeomorphic differential structures on a manifold is not possible for metrisable manifolds in dimension up to 3, \cite{M11}: this result is due to Rad\'o in dimensions 1 and 2 and Moise in dimension 3.

On the other hand in \cite{N92} it is shown that when we relax the metrisability condition then even in dimension 1 there are $2^{\aleph_1}$ mutually non-diffeomorphic differential structures (on the long ray, hence also the long line $\L$). As a result there are also $2^{\aleph_1}$ mutually non-diffeomorphic differential structures on the long plane $\L^2$. In this paper we address the question: does the long plane support differential structures which are not diffeomorphic to any product structure? Our answer is ``yes.'' 

As usual we denote by  $\omega_1$ the set of countable ordinals with the order topology. Let $\mathbb L_{\ge0}$ denote the \emph{closed long ray}, ie the set $\omega_1\times[0,1)$ with the lexicographic order topology, and let $\mathbb L$ denote the \emph{long line} which is obtained from two copies of the closed long ray with their initial points identified to $0$. The \emph{(open) long ray} is the 1-manifold $\L_+=\L_{\ge0}-\{(0,0)\}$. Identify $\al\in\om_1$ with $(\al,0)\in\L_{\ge0}$. We will exhibit non-product differential structures on $\L_+^2$. As in \cite{N92} similar structures may then be deduced on $\L^2$.

The following result is well-known and is found in many books introducing Set Theory but we include it for completeness. Note that it does not matter whether we are considering $C$ and $D$ as subsets of $\om_1$ or $\L_+$.

\begin{lem} \label{lemmaclub}
If $C,D\subset\omega_1$ are closed unbounded subsets then $C\cap D$ is also closed and unbounded.
\end{lem}

\begin{prop} \label{shrinkL}
Suppose that $\mc D$ is a differential structure on $\L_+$ and $\al\in\L_+$. Then $((\al,\om_1),\mc D|(\al,\om_1))$ is diffeomorphic to $(\L_+,\mc D)$.
\end{prop}
\begin{proof}
Choose $\be,\ga\in\L_+$ such that $\al<\be<\ga$. Because $\R$ has a unique differential structure up to diffeomorphism and $(0,\ga)\subset\L_+$ is homeomorphic to $\R$ we may choose a diffeomorphism $g:((0,\ga),\mc D|(0,\ga))\to((0,3),\mc U)$, where $\mc U$ is the usual differential structure on $\R$ restricted to $(0,3)$. Furthermore we may assume that $g(\al)=1$ and $g(\be)=2$. Next let $\th:(0,3)\to(1,3)$ be a diffeomorphism (relative to $\mc U$) such that $\th(t)=t$ for all $t\in[2,3)$. For example set $\ds\th(t)=\left\{\begin{array}{lll}t+\sqrt{e}e^\frac{1}{t-2} & \mbox{ if } & t<2\\ t & \mbox{ if } & t\ge 2 \end{array}\right..$ Now define $h:\L_+\to(\al,\om_1)$ by $\ds h(t)=\left\{\begin{array}{lll}g^{-1}\th g(t) & \mbox{ if } & t<\ga\\ t & \mbox{ if } & t>\be \end{array}\right..$ Then $h$ is a diffeomorphism with respect to the structure $\mc D$.
\end{proof}

Recall the following result from \cite[Corollary 2.6]{BDG10}.
\begin{prop} \label{lineup} Suppose that $h:\L_+^2\to\L_+^2$ is an orientation-preserving homeomorphism. Then $\{\al\in\om_1\ /\ h(\L_+\*\{\al\})=\L_+\*\{\al\}\}$ is a closed unbounded set.
\end{prop}

\begin{cor} \label{nice submanifolds} Suppose that $\mc F$ is a differential structure on $\L_+^2$ and that $\mc F$ is diffeomorphic to the product of two structures. Then 
$$\{\al\in\om_1\ /\ \L_+\*\{\al\} \mbox{ is a differentiable submanifold of } (\L_+^2,\mc F)\}$$
and
$$\{\al\in\om_1\ /\ \{\al\}\*\L_+ \mbox{ is a differentiable submanifold of } (\L_+^2,\mc F)\}$$
are closed unbounded subsets of $\om_1$. 
\end{cor}
\begin{proof}
There are two differential structures, say $\mc D,\mc E$, on $\L_+$ and a diffeomorphism $h:(\L_+,\mc D)\*(\L_+,\mc E)\to(\L_+^2,\mc F)$. Interchanging the roles of $\mc D$ and $\mc E$ if necessary we may assume $h$ preserves orientation. By Proposition \ref{lineup}, $S=\{\al\in\om_1\ /\ h(\L_+\*\{\al\})=\L_+\*\{\al\}\}$ is a closed unbounded set. Thus $\L_+\*\{\al\}$ is a differentiable submanifold of $(\L_+^2,\mc F)$ for each $\al\in S$. Interchanging the coordinates leads to the other half.
\end{proof}

We also require the following folklore result, cf \cite[Theorem 1]{M64} and \cite[Theorem 3 page 46]{M77}. 
\begin{prop}\label{extend diffeomorphism}
Let $M\subset\R^2$ be a compact topological manifold with boundary, $K\subset M$ a compact subset which contains the boundary of $M$ and suppose that $h:M\to M$ is a homeomorphism which is a diffeomorphism on a neighbourhood of $K$. Then $h$ can be approximated arbitrarily closely by a homeomorphism which is a diffeomorphism on $\Int{M}$ and agrees with $h$ on a neighbourhood of $K$.
\end{prop}

\section{Exotic Differential Structures on $\L_+^2$}

We now present a method of constructing from two differential structures on the long ray a differential structure on $\L_+^2$ which is not diffeomorphic to the product of any two differential structures on $\L_+$. The construction allows us to verify that there are $2^{\aleph_1}$ many non-diffeomorphic such structures. 

We require an auxiliary shearing homeomorphism $\si:[0,5]^2\to[0,5]^2$. The homeomorphism $\si$ is the identity except in the rectangle $\left(3,4\right)\*\left(1,4\right)$, does not change the first coordinate and maps the straight line segment $\left[3,4\right]\*\left\{2\right\}$ onto the two line segments $\left\{\left(x,3-2\left|x-3\frac{1}{2}\right|\right)\ :\ 3\le x\le4\right\}$. The notation $I_\al=(0,\al+1)$, $\ol{I_\al}=[0,\al+1]$, $O_\al=I_\al^2\subset\L_+^2$ and  $\ol{O_\al}=\ol{I_\al}^2\subset\L_{\ge0}^2$ is fixed throughout this section.

Begin with two differential structures $\mc D$ and $\mc E$ on $\L_+$; for example any of those in \cite{N92} will do. For each $\al\in\om_1\sm\{0\}$ choose order-preserving homeomorphisms $\psi_\al,\chi_\al:\ol{I_\al}\to[0,5]$ so that $\psi_\al(\al)=\chi_\al(\al)=2$, and that $(I_\al,\psi_\al)\in\mc D$ and $(I_\al,\chi_\al)\in\mc E$. 

For each $\al\in\om_1\sm\{0\}$ we will construct by induction on $\al$ a homeomorphism $\ph_\al:\ol{O_\al}\to [0,5]^2$ in such a way that $\{(O_\al,\ph_\al)\ /\ \al\in\om_1\sm\{0\}\}$ is a basis for a differential structure on $\L_+^2$, i.e., for each $\al,\be\in\om_1\sm\{0\}$ the maps $\ph_\al\ph_\be^{-1}$ and $\ph_\be\ph_\al^{-1}$ are smooth where defined within $(0,5)^2$. The induction includes the further condition:
\begin{itemize}
\item The homeomorphisms $\psi_\al\*\chi_\al$ and $\ph_\al$ agree on neighbourhoods of  $\ol{O_\al}\sm O_\al$ and of $[0,\al)\*[\al,\al+1]$ as well as on a neighbourhood in $\ol{O_\al}\sm[0,\al)^2$ of $\{\al\}\*[0,\al]$.
\end{itemize}

\begin{figure}[ht]  
\begin{picture}(320,160)(0,0)
\put(40,10){ \epsfig{figure=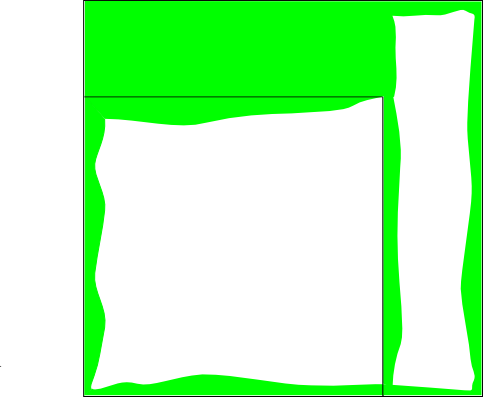,width=70mm}}
\put(250,100){$\ol{O_\al}$}
\put(75,0){$0$}
\put(198,0){$\al$}
\put(230,0){$\al+1$}
\put(67,7){$0$}
\put(67,131){$\al$}
\put(50,170){$\al+1$}
 \end{picture}
 \caption{\label{h_n} Where $\ph_\al=\psi_\al\*\chi_\al$.}
 \end{figure}

\noindent{\bf Definition of $\ph_1$:-} Set $\ph_1=\si(\psi_1\*\chi_1)$.\\
{\bf Definition of $\ph_{\al+1}$ given $\ph_\al$:-} Define
$$\ph_{\al+1}(z)=\left\{\begin{array}{lll} (\psi_{\al+1}\*\chi_{\al+1})(\psi_\al\*\chi_\al)^{-1}\ph_\al(z) & \mbox{ if } & z\in \ol{O_\al};\\  \si(\psi_{\al+1}\*\chi_{\al+1})(z) & \mbox{ if } & z\in \ol{O_{\al+1}}-O_\al\end{array} \right..$$  
It is easily checked that the inductive conditions are satisfied.

\noindent{\bf Definition of $\ph_\al$, where $\al$ is a limit ordinal, given $\ph_\be$ for all $\be\in\om_1\sm\{0\}$ with $\be<\al$:-} Firstly choose some metric $d$ on $\ol{O_\al}$ compatible with the topology. Next choose an increasing sequence $\la\al_n\ra$ from $\om_1\sm\{0\}$ converging to $\al$; set $\al_0=0$. Somewhat as in the previous case we would like to let $\ph_\al$ be $(\psi_{\al}\*\chi_{\al})(\psi_{\al_n}\*\chi_{\al_n})^{-1}\ph_{\al_n}$ on $\ol{O_{\al_n}}$ and be $\si(\psi_{\al}\*\chi_{\al})$ outside the union of these common domains but this would work only if all maps of the form $(\psi_{\al_m}\*\chi_{\al_m})^{-1}\ph_{\al_m}$ and $(\psi_{\al_n}\*\chi_{\al_n})^{-1}\ph_{\al_n}$ agree on $\ol{O_{\al_{\min\{m,n\}}}}$. We modify these maps inductively so that they do agree, at least on enough of $\ol{O_{\al_{\min\{m,n\}}}}$. To effect this we construct a sequence of homeomorphisms $\la h_n:[0,3]^2\to[0,3]^2\ra$, where $n\ge1$. We demand the following properties:
\begin{itemize}
\item $h_n:(0,3)^2\to(0,3)^2$ is a diffeomorphism;
\item $h_n$ is the identity on a neighbourhood of $([0,3]\*[2,3])\cup([2,3]\*[0,3])$;
\item $h_n=(\psi_{\al_n}\*\chi_{\al_n})(\psi_{\al_{n-1}}\*\chi_{\al_{n-1}})^{-1}h_{n-1}\ph_{\al_{n-1}}\ph_{\al_n}^{-1}$ on a neighbourhood of $\ph_{\al_n}([0,\al_{n-1})^2)$ when $n>1$;
\item $h_n=(\psi_{\al_n}\*\chi_{\al_n})\ph_{\al_n}^{-1}$ on a neighbourhood of $\ph_{\al_n}([0,\al_{n-1})\*[\al_{n-1},\al_n])$ when $n>1$;
\item for $n>1$, on $\ph_{\al_n}([\al_{n-1},\al_n]\*[0,\al_{n-2}])$, $h_n$ is sufficiently close to $(\psi_{\al_n}\*\chi_{\al_n})\ph_{\al_n}^{-1}$ that for any $(x,y)\in\ph_{\al_n}([\al_{n-1},\al_n]\*[0,\al_{n-2}])$, we have
$$d\left((\psi_{\al_n}\*\chi_{\al_n})^{-1}(x,y),(\psi_{\al_n}\*\chi_{\al_n})^{-1}h_n(x,y)\right)<\frac{1}{n}.$$
To achieve this we use uniform continuity of $(\psi_{\al_n}\*\chi_{\al_n})^{-1}$.
\end{itemize}
Notice that the conditions on $h_n$ are mutually consistent by the inductively assumed conditions on $\ph_{\al_m}$ and $h_{n-1}$.

\begin{figure}[ht]  
\begin{picture}(320,210)(0,0)
\put(60,10){ \epsfig{figure=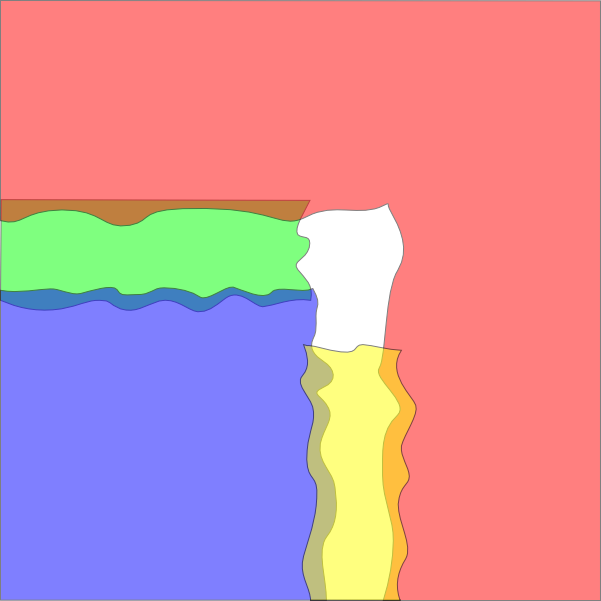,width=70mm}}
\put(150,180){\footnotesize$h_n=$ identity}
\put(70,120){\footnotesize$h_n=(\psi_{\al_n}\*\chi_{\al_n})\ph_{\al_n}^{-1}$}
\put(65,80){\footnotesize$h_n=(\psi_{\al_n}\*\chi_{\al_n})$}
\put(85,65){\footnotesize$(\psi_{\al_{n-1}}\*\chi_{\al_{n-1}})^{-1}$}
\put(85,50){\footnotesize$h_{n-1}\ph_{\al_{n-1}}\ph_{\al_n}^{-1}$}
\put(140,30){\footnotesize$h_n\sim(\psi_{\al_n}\*\chi_{\al_n})\ph_{\al_n}^{-1}$}
\put(62,0){$0$}
\put(192,0){$2$}
\put(257,0){$3$}
\put(55,7){$0$}
\put(55,140){$2$}
\put(55,205){$3$}
 \end{picture}
 \caption{\label{h_n} Constraints on $h_n$.}
 \end{figure}

Let $h_1:[0,3]^2\to[0,3]^2$ be the identity. Suppose given $n>1$ such that $h_{n-1}$ has been defined. By Proposition \ref{extend diffeomorphism} there is a homeomorphism $h_n:[0,3]^2\to[0,3]^2$ satisfying the requirements.

 Now define $\ph_\al$ by

$$\ph_\al(x)=\left\{\begin{array}{lll} (\psi_{\al}\*\chi_{\al})(\psi_{\al_n}\*\chi_{\al_n})^{-1}h_n\ph_{\al_n}(x) & \mbox{if} & x\in[0,\al_n]^2 \mbox{ for some } n\\  \si(\psi_{\al}\*\chi_{\al})(x) & \mbox{if} & x\in \ol{O_\al}\setminus(0,\al)^2 \end{array}\right..$$
The function $\ph_\al$ is well-defined because  if $m<n$ then $(\psi_{\al_m}\*\chi_{\al_m})^{-1}h_m\ph_{\al_m}$ and\linebreak $(\psi_{\al_n}\*\chi_{\al_n})^{-1}h_n\ph_{\al_n}$ agree on $[0,\al_m]^2$. It is easily verified that $\ph_\al$ is a homeomorphism, the main problem being to verify continuity on $\{\al\}\*[0,\al]$. It is here that we require precision in the approximation of the homeomorphism $(\psi_{\al_n}\*\chi_{\al_n})\ph_{\al_n}^{-1}$ by the diffeomorphism as required in the last inductive assumption for $h_n$. The approximation must improve as $n$ increases so that any sequence $\la(x_n,y_n)\ra$ in $[0,\al)^2$ converging to $(\al,y)\in\{\al\}\*[0,\al]$ is mapped by $(\psi_{\al_m}\*\chi_{\al_m})^{-1}h_m\ph_{\al_m}$ ($m$ increasing with $n$) to a sequence which also converges to $(\al,y)$. Then\linebreak $\ph_\al(x_n,y_n)\to(\psi_{\al}\*\chi_{\al})(\al,y)=\ph(\al,y)$ as $\si$ is the identity on $\{2\}\*[0,5]$.

Suppose $\be<\al$. Then the coordinate transformation between $\ph_\al$ and $\ph_\be$ is smooth on $\ph_\be(O_{\be+1})$ because 
$$\ph_\al\ph_\be^{-1}=(\psi_{\al}\*\chi_{\al})(\psi_{\al_{n}}\*\chi_{\al_{n}})^{-1}h_{n}\ph_{\al_n}\ph_\be^{-1}$$
is a composition of coordinate transition functions together with the diffeomorphism $h_n$ and hence is smooth, where $n$ is chosen so that $\al_n>\be$. Similarly its inverse is smooth. 

The remaining condition demanded of $\ph_\al$ is also satisfied.

Thus we have constructed a basis $\{(O_\al,\ph_\al)\ /\ \al\in\om_1\sm\{0\}\}$ for a differential structure on $\L_+^2$. Call this structure $\mc F$.

\begin{claim}
The differential structure $\mc F$ is not diffeomorphic to a product of structures on $\L_+$.
\end{claim}
\begin{proof}
Let $\al<\om_1$ be any non-zero ordinal. We first show that $\L_+\*\{\al\}$ is not a smooth submanifold of $(\L_+^2,\mc F)$. Suppose instead that $\L_+\*\{\al\}$ is a smooth submanifold of $(\L_+^2,\mc F)$. Then $(\al,\al+1)\*\{\al\}$ is also a smooth submanifold of $(\L_+^2,\mc F)$, so there is a chart $\Big((\al,\al+1)\*(0,\al+1),\ph\Big)\in\mc F$ such that $\ph^{-1}(\R\*\{0\})=(\al,\al+1)\*\{\al\}$. It follows that $\ph_\al\ph^{-1}(\R\*\{0\})=\ph_\al((\al,\al+1)\*\{\al\})$ is a smooth submanifold of $\R^2$ with the usual differential structure. However, for $t\in(\al,\al+1)$ we have $\ph_\al(t,\al)=\si(\psi_\al(t),2)$, and hence 
$$\ph_\al((\al,\al+1)\*\{\al\})=\left(\left(2,3\right]\cup\left[4,5\right)\right)\*\{2\}\cup\left\{\left(x,3-2\left|x-3\frac{1}{2}\right|\right)\ :\ 3\le x\le4\right\}.$$
As this last set is not a smooth submanifold of $\R^2$, it follows that $(0,\om_1)\*\{\al\}$ is not a smooth submanifold of $(\L_+^2,\mc F)$.

The claim now follows from Lemma \ref{lemmaclub} and Corollary \ref{nice submanifolds} because $\om_1\sm\{0\}$ is closed and unbounded.
\end{proof}

We now address the question: how many exotic differential structures does $\L_+^2$ support? The argument presented in \cite[p.156]{N92} that $\L_+$ supports no more than $2^{\aleph_1}$ many mutually non-diffeomorphic differential structures applies as well to $\L_+^2$. On the other hand \cite[Theorem 5.2]{N92} exhibits exactly $2^{\aleph_1}$ many mutually non-diffeomorphic differential structures on $\L_+$. Thus we might expect to find $2^{\aleph_1}$ many mutually non-diffeomorphic exotic differential structures on $\L_+^2$, and this is indeed the case.

Let $\mc D$ be any differential structure on $\L_+$. Apply the construction above with $\mc E=\mc D$ and denote the resulting exotic differential structure on $\L_+^2$ by $\widehat{\mc D}$.
\begin{thm}\label{number of exotic structures}
There are $2^{\aleph_1}$ mutually non-diffeomorphic exotic differential structures on $\L_+^2$.
\end{thm}
\begin{proof}
Suppose given differential structures $\mc D$ and $\mc E$ on $\L_+$ and an orientation-preserving diffeomorphism $h:(\L_+^2,\widehat{\mc D})\to(\L_+^2,\widehat{\mc E})$. By Proposition \ref{lineup}, for a closed unbounded set of $\al\in\om_1$, the map $h$ restricts to a homeomorphism taking $\{\al\}\*(\al,\om_1)$ to itself. Now $\widehat{\mc D}|\{\al\}\*(\al,\om_1)$ is the same as $\mc D\*\mc D|\{\al\}\*(\al,\om_1)$ with the same for $\mc E$ so, using Proposition \ref{shrinkL} and denoting ``is diffeomorphic to'' by $\approx$, we have
\begin{eqnarray*}
(\L_+,\mc D) & \approx & ((\al,\om_1),\mc D)\\
& \approx & (\{\al\}\*(\al,\om_1),\mc D\*\mc D|\{\al\}\*(\al,\om_1))\\
& \approx & (\{\al\}\*(\al,\om_1),\widehat{\mc D}|\{\al\}\*(\al,\om_1))\\
& \approx & (\{\al\}\*(\al,\om_1),\widehat{\mc E}|\{\al\}\*(\al,\om_1))\\
& \approx & (\{\al\}\*(\al,\om_1),\mc E\*\mc E|\{\al\}\*(\al,\om_1))\\
& \approx & ((\al,\om_1),\mc E)\\
& \approx & (\L_+,\mc E).
\end{eqnarray*}

It follows that for any differential structure $\mc D$ on $\L_+$ there can be at most one equivalence class of structures, represented say by $\mc E$, such that $(\L_+^2,\widehat{\mc D})$ is diffeomorphic to $(\L_+^2,\widehat{\mc E})$ but $(\L_+,\mc D)$ is not diffeomorphic to $(\L_+,\mc E)$. Indeed, if $\mc D$, $\mc E$ and $\mc F$ are three differential structures on $\L_+$ and $(\L_+^2,\widehat{\mc D})$ is diffeomorphic to both $(\L_+^2,\widehat{\mc E})$ and $(\L_+^2,\widehat{\mc F})$ but $(\L_+,\mc D)$ is not diffeomorphic to either $(\L_+,\mc E)$ or $(\L_+,\mc F)$, then diffeomorphisms $g:(\L_+^2,\widehat{\mc D})\to(\L_+^2,\widehat{\mc E})$ and $h:(\L_+^2,\widehat{\mc D})\to(\L_+^2,\widehat{\mc F})$ must reverse orientation. In that case the diffeomorphism $hg^{-1}:(\L_+^2,\widehat{\mc E})\to(\L_+^2,\widehat{\mc F})$ preserves orientation and hence $(\L_+,\mc E)$ is diffeomorphic to $(\L_+,\mc F)$ by what we have already shown.

It now follows from \cite[Theorem 5.2]{N92} that there are $2^{\aleph_1}$ mutually non-diffeomorphic exotic differential structures on $\L_+^2$.
\end{proof}

As a complement to Theorem \ref{number of exotic structures} we have the following.

\begin{thm}
There are $2^{\aleph_1}$ mutually non-diffeomorphic product differential structures on $\L_+^2$.
\end{thm}
\begin{proof}
It suffices to show that if $\mc D$ and $\mc E$ are two differential structures on $\L$ such that $(\L^2,\mc D\*\mc D)$ is diffeomorphic to $(\L^2,\mc E\*\mc E)$ then $(\L,\mc D)$ is diffeomorphic to $(\L,\mc E)$. However it is easy to show that the homeomorphism $(x,x)\mapsto x$ from the diagonal $\Delta$ of $\L^2$ to $\L$ is a diffeomorphism from $(\Delta,\mc D\*\mc D|\Delta)$ to $(\L,\mc D)$.
\end{proof}


\end{document}